\documentclass[12pt,reqno]{amsart}
\usepackage{amsmath,amssymb,hyperref,mathrsfs,graphicx}
\usepackage[utf8x]{inputenc}
\usepackage{amsmath}
\usepackage{amsfonts}
\usepackage{latexsym}
\usepackage{amsthm}
\usepackage{amssymb,amscd}
\usepackage{xargs}
\usepackage{graphicx}
\usepackage{color}
\usepackage{enumitem}
\usepackage[foot]{amsaddr}

\usepackage{lmodern}

\usepackage[text={5.7in, 8in}, centering]{geometry}
\usepackage{microtype}

\usepackage{todonotes}

\usepackage{hyperref}

\newtheorem{thm}{Theorem}[section]

\newtheorem{lem}[thm]{Lemma}

\newtheorem{prop}[thm]{Proposition}
\newtheorem{cor}[thm]{Corollary}

\theoremstyle{remark}
\newtheorem*{rmk}{Remark}

\newcommand{\be}{\begin{eqnarray}}
\newcommand{\ee}{\end{eqnarray}}
\newcommand{\beal}{\begin{aligned}}
\newcommand{\eeeal}{\end{aligned}}
\newcommand{\beit}{\begin{itemize}}
\newcommand{\eeit}{\end{itemize}}
\newcommand{\blm}{\begin{lem}}
\newcommand{\elm}{\end{lem}}

\newcommand{\T}{\mathbb{T}}
\newcommand{\R}{\mathbb{R}}

\newcommand{\C}{\mathbb{C}}

\newcommand{\N}{\mathbb{N}}
\newcommand{\Z}{\mathbb{Z}}

\newcommand{\cB}{\mathcal{B}}

\newcommand{\cD}{\mathcal{D}}

\newcommand{\cA}{\mathcal{A}}
\newcommand{\bP}{\mathbf{P}}
\newcommand{\bQ}{\mathbf{Q}}

\renewcommand{\Im}{\mathrm{Im}\,}
\renewcommand{\Re}{\mathrm{Re}\,}
\newcommand{\id}{\mathrm{id}}
\newcommand{\tdr}{\tilde{r}}

\newcommand{\olu}{\overline{u}}
\newcommand{\ulu}{\underline{u}}

\newcommand{\nek}{{\mathrm{nek}}}
\newcommand{\laz}{\mathrm{laz}}
\newcommand{\app}{\mathrm{app}}
\newcommand{\cau}{\mathrm{cau}}

\newcommand{\bmat}[1]{\begin{bmatrix} #1 \end{bmatrix}}

\subjclass[2010]{Primary 37J40, Secondary 37J50, 70H08}
\keywords{billiard map, rational caustic, KAM theorem, Nekhoroshev theorem}

\title{Density of convex billiards with rational caustics}

\author{Vadim Kaloshin$^1$}
\address{$^1$ Department of Mathematics, University of Maryland}
\author{Ke Zhang$^2$}
\address{$^2$ Department of Mathematics, University of Toronto}

\begin{document}
\maketitle
\begin{abstract}
	We show that in the space of all convex billiard boundaries, the set of boundaries with rational caustics is dense. More precisely, the set of billiard boundaries with caustics of rotation number $1/q$ is polynomially sense in the smooth case, and exponentially dense in the analytic case. 
\end{abstract}

\section{Introduction}

Let $\Omega\subset \R^2$ be a strictly convex billiard table, assume that its boundary is given by 
\[
	\partial \Omega = 
	\left\{ r(s) \in \R^2 \right\}, \quad r(s) = r(s + 1),\,  s \in \R.
\]
By applying a translation, we can assume that $0 \in \Omega$. Let $\kappa(s)$ denote the curvature $\partial\Omega$ at $r(s)$, we assume for $D > 1$:
\begin{equation}
  \label{eq:real-assumption}
  	|\dot{r}(s)| \ge D^{-1}, \quad \kappa(s) \ge D^{-1}. \tag{A1} 
\end{equation}

We are interested in the existence of rational caustics of rotation number $1/q$. Baryshnikov and Zharnitsky (\cite{BZ2006}) showed that boundaries admitting caustics of a fixed rational rotation number is a finite co-dimension sub-manifold among all boundaries. Unlike irrational caustics which tend to be robust under perturbation (due to KAM theorem, see \cite{Laz1973}), rational caustics tend to break up under perturbation. In this paper, we investigate how abundant billiards with rational caustics are,  within the space of all boundaries. Analog to Hamiltonian averaging suggests the density of such billiards should be polynomial in $q$ in the smooth case, and exponential in $q$ in the analytic case. Indeed, Mart\'{i}n, Ram\'{i}rez-Ros and Tamarit-Sariol showed (\cite{MRT2016}), in the analytic case, Mather's $\Delta W_{p/q}$ function is exponentially small. This means that the billiard is ``exponentially close'' to having a caustic. This, however, does not mean one can perturb the boundary by an exponentially small amount to create a caustic. The reason is the billiard dynamics depends rather implicitly on the boundary, and it is not obvious how to perturb the boundary to obtain the desired caustics. 

Our main result is that this indeed can be done. 
\begin{thm}\label{thm:main}
\begin{enumerate}
 \item  Suppose $r$ is a real analytic boundary, extensible to a complex neighborhood of size $\sigma_0$.  Then there is $C>1$ depending only on $\sigma_0, D$ and the analytic norm of $r$, such that for $q$ sufficiently large, there a real analytic boundary $r_\cau$ admitting a caustic of rotation number $1/q$, satisfying 
 \[
 	\|r - r_\cau\|_{\sigma_0/8} < C e^{- C^{-1} q},
 \] 
 there $\|\cdot\|_{\sigma}$ is the supremum norm of the complex extension. 
 \item Suppose $r$ is $C^m$ with $m > 4$,  then for each $3< l < m$,  there is $C > 1$ depending only on $D$, $\|r\|_{C^m}$, $l$, $m$, and a $C^\infty$ boundary $r_\cau$ admitting a caustic of rotation number $1/q$, such that 
 \[
 	\|r - r_\cau\|_{C^l} < C q^{\frac{m-l}{7}}. 
 \]
\end{enumerate}
\end{thm}
\begin{rmk}
The estimates obtained are not optimal. In part (1), it is reasonable to study how the constant $C$ depend on shrinking of analytic width. In part (2) the exponent certainly has room for improvement. Since the paper is already quite technical, we opt for simplicity of proofs rather than strength of the result. 
\end{rmk}

The billiard problem has a caustic of rotation number $1/q$ if there exists a homeomorphism $u: \R \to \R$ satisfying $u(\theta +1) = u(\theta) + 1$, such that 
\begin{equation}
	\label{eq:inv-cur}
	E(r, u) := \partial_{s'} L(r, u(\theta - \alpha), u(\theta)) + \partial_s L(r, u(\theta), u(\theta + \alpha))=0, 
\end{equation}
where 
\[
	L(r; s, s') = |r(s') - r(s)|, \quad  \alpha = 1/q.
\]
This can be viewed as the ``Lagrangian setting'' compared to the ``Hamiltonian setting'' of the billiard map. In (\cite{LM2001}),  a KAM theorem is proved using this setting. 

To obtain a solution to $E(r, u) = 0$, we first perform coordinate changes to obtain approximate solutions. This is done in two steps: First we use a Lazutkin-type (\cite{Laz1973}) normal form to convert the billiard to a map close to rigid rotation. We then apply a Nekhoroshev-type (see \cite{LN1992} for more background) normal form, which in the analytic case, proves the existence of an approximate caustic with only exponential error. Note that this step is in fact done in \cite{MRT2016}, however, their result does not provide the quantitative estimates depending only on uniform conditions on the boundary. We provide an alternative approach which only uses the Lagrangian setting. 

Suppose $(r, u)$ is an approximate solution to \eqref{eq:inv-cur}, we seek  a true solution $E(\tdr, \tilde{u}) = 0$ close to $(r, u)$. Given any function $g: \T \to \R$ with $g(\theta) = \sum_{k \in \Z} g_k e^{2 \pi i k \theta}$, we define 
\[
	[g]_q = \sum_{k \in q\Z} g_k e^{2\pi k \theta} = \frac{1}{q}\sum_{i =1}^q g(\theta + i \alpha), \quad \{g\}_q = g - [g]_q.  
\]
We call a function $g$ resonant if $[g]_q = g$, and non-resonant if $[g]_q = 0$. 

Our strategy is to first deform $r$ so that the function $u_\theta E(r, u)$ has zero resonant component, then we adjust $u$ to reduce the non-resonant component. For the first step, we show in Corollary~\ref{cor:proj}  that for each $(r, u)$ there is $a: \T \to \R$ such that 
\[
	[u_\theta E(e^a r, u)]_q = 0. 
\]

Therefore, we may assume  $[u_\theta E(r, u)]_q = 0$. In this case, the method of Moser and Levi (\cite{LM2001}) provides a solution to 
\[
	u_\theta \partial_u E(r, u) \cdot v = - u_\theta E(r, u) - v \frac{d}{d\theta} E(r, u),
\]
which allows a KAM-type iteration to find a solution. Moreover, due to the identity 
\[
	E(r \circ u, \id) = u_\theta E(r, u), 
\]
we can perform the iteration at $u = \id$. 

The outline of this paper is as follows. In Section~\ref{sec:comp}, we construct the projection to non-resonant space, and recall the Moser-Levi algorithm. In Section~\ref{sec:analytic-est}, we perform basic estimates in the analytic norm. The Nekhoroshev averaging is performed in Section~\ref{sec:nek}, and the Lazutkin normal form is in Section~\ref{sec:laz} with some details deferred to the appendix. The KAM iteration is done in Section~\ref{sec:kam}, where we also prove the main theorem. 

\section{Basic computations}
\label{sec:comp}

Let us use the following notations:
\begin{enumerate}
	\item Denote $u(\theta)$ by $u$ when there is no confusion, we write $u^- = u(\theta  -\alpha)$, $u^+ = u(\theta + \alpha)$.
	\item For any function $f$ of $s$, $\Delta f$ denotes $f(u^+) - f(u)$. Under the same convention, $\Delta f^-$ denotes $f(u) - f(u^-)$. Note that the composition with $u$ is implied whenever $\Delta$ notation is used. 
\end{enumerate}
Then
\begin{equation}
	\label{eq:E}
	E(r, u)  = \partial_{s'} L(r, u^-, u) + \partial_{s} L(r, u, u^+) 
	= \left\langle   
	\frac{\Delta r^-}{|\Delta r^-|} - \frac{\Delta r}{|\Delta r|} , \dot{r}(u) 
	\right\rangle .
\end{equation}

\begin{lem}\label{lem:prjection}
Let $a: \T \to \R$ be such that $a(u^+) = a(u)$. Then 
\[
	E(e^{a}r, u) = e^{a(u)} \left( \dot{a}(u) F(r, u) +  E(r, u) \right), 
\]
where 
\[
	F(r, u) = \left\langle
	\frac{\Delta r^-}{|\Delta r^-|} - \frac{\Delta r}{|\Delta r|} , r(u)
	\right\rangle. 
\]
\end{lem}
\begin{proof}
	Write $\tdr = e^a r$, then
	\[
		\Delta \tdr = e^{a(u^+)} r(u^+)  - e^{a(u)} r(u) = 
		e^{a(u)}\Delta r, 
	\]
	as a result, $\Delta \tdr / |\Delta \tdr| = \Delta r / |\Delta r|$. Since 
	\[
		\frac{d}{ds} \tdr = e^a\left(  \dot{a} r +  \dot{r} \right), 
	\]
	from \eqref{eq:E} we get 
	\[
		E(\tdr, u) = \dot{a}(u) e^{a(u)} \left\langle  \frac{\Delta r^-}{|\Delta r^-|} - \frac{\Delta r}{|\Delta r|} ,  r \right\rangle +  e^{a(u)} \left\langle  \frac{\Delta r^-}{|\Delta r^-|} - \frac{\Delta r}{|\Delta r|} , \dot{r} \right\rangle. 
	\]
\end{proof}

\begin{lem}\label{lem:F-const}
The function $F(r, u)$ has the following special property:
\[
	[F(r, u)]_q = [|\Delta r|]_q, 
	\quad 
	\frac{d}{d\theta}[F(r, u)]_q =  [u_\theta E(r, u)]_q. 
\]
\end{lem}
\begin{proof}
	We have 
	\[
		\begin{aligned}
			{[F(r, u)]_q}  &= \frac{1}{q} \sum_{k = 1}^q \left \langle \frac{\Delta r^-}{|\Delta r^-|} , r \right \rangle \circ (\theta + k \alpha) -
			\frac{1}{q} \sum_{k = 1}^q \left \langle \frac{\Delta r}{|\Delta r|} , r \right \rangle \circ (\theta + k \alpha) \\
			& = \frac{1}{q} \sum_{k = 1}^q \left \langle \frac{\Delta r}{|\Delta r|} , r^+ \right \rangle \circ (\theta + k \alpha) -
			\frac{1}{q} \sum_{k = 1}^q \left \langle \frac{\Delta r}{|\Delta r|} , r \right \rangle \circ (\theta + k \alpha) \\
			& = \frac{1}{q} \sum_{k = 1}^q \left \langle \frac{\Delta r}{|\Delta r|} , \Delta r \right \rangle \circ (\theta + k \alpha) = \frac{1}{q} \sum_{k = 1}^q |\Delta r| \circ  (\theta + k \alpha). 
		\end{aligned}
	\]
	The second formula follows a standard computation. 
\end{proof}

\begin{cor}\label{cor:proj}
Given $(r, u)$ let $a$ be the unique solution to 
\begin{equation}
	\label{eq:a}
	\frac{d}{d\theta} a(u) = \dot{a}(u) u_\theta = - \frac{[u_\theta E(r, u)]_q}{[F(r, u)]_q}
\end{equation}
under the periodic condition $a(u^+) = a(u)$. 
Then for $\tdr = e^a r$, we have 
\[
	[u_\theta E(\tdr, u)]_q = 0. 
\]
\end{cor}
\begin{proof}
	We first need to show that \eqref{eq:a} defines a function satisfying $a(u^+) = a(u)$. By Lemma~\ref{lem:F-const} we have 
	\[
		a\circ u(\theta + \alpha) - a \circ u(\theta) = \int_{\theta}^{\theta + \alpha} \frac{d}{d \tau} \log [F(r, u)](\tau) d\tau = \log[F(r, u)]\bigr|_\theta^{\theta + \alpha} = 0. 
	\]
	Using Lemma~\ref{lem:prjection}, and noting that $a(u)$ and $\dot{a}(u) \cdot u_\theta$ are $\alpha$ periodic, we have 
	\[
		[u_\theta E(\tdr, u)]_q = e^{a(u)} \left( \dot{a}(u) u_\theta [F(r, u)]_q + [u_\theta E(r, u)]_q \right) = 0. 
	\]
\end{proof}

We now compute the linearized operator in $u$. Following \cite{LM2001}, for a function $g(\theta)$, define 
\[
	\nabla g(\theta) = g(\theta + \alpha) - g(\theta), \quad 
	\nabla^- g(\theta) = g(\theta) - g(\theta - \alpha), 
\]
and write 
\[
	L_{12}(r, u) = \partial^2_{ss'}L(r, u, u^+). 
\]

\begin{lem}\cite{LM2001}
We have 
\begin{equation}
	\label{eq:moser-levi}
	u_\theta \partial_u E(r, u)\cdot v + v \frac{d}{d\theta} E(r, u) = 
	\nabla^-\left( L_{12} u_\theta u_\theta^+ \nabla w \right),
\end{equation}
where $w = v/u_\theta$. 
\end{lem}

\section{Estimates of the analytic norm}
\label{sec:analytic-est}

In this section we introduce the function space and provide basic estimates in the analytic norm. 

For $\sigma >0$, we define $\cA_\sigma$ to be the set of bounded complex analytic functions on  the set $\T_\sigma = \{|\Im \theta| < \sigma\} \in \C/\Z$ which takes real values for real $\theta$ (namely $\overline{f(\theta)} = f(\overline{\theta})$). It is a Banach space with the norm given by 
\[
	\|f\|_\sigma = \sup_{\theta \in \T_\sigma} |f(\theta)|.
\]
For $l \in \N$, define 
\[
	\cA_{\sigma, l} = \left\{ f: \T_\sigma \to \C; \quad \overline{f(\theta)} = f(\overline{\theta}), \quad  \|f\|_{\sigma, l} := \sup_{0 \le k \le l} \|f^{(k)}(\theta)\|_\sigma < \infty \right\}. 
\]
We will also consider the space $\cD_{\sigma_1, \sigma_2}$ of real analytic functions on $\T \times (0, \sigma_2)$ extensible to $\T_{\sigma_1} \times B_{\sigma_2}$, where $B_\sigma = \{h: \|h\| < \sigma\}$ is the complex ball of radius $\sigma$. Similarly define the space $\cD_{\sigma_1, \sigma_2, l}$, the norms $\|\cdot\|_{\sigma_1, \sigma_2}$ and $\|\cdot\|_{\sigma_1, \sigma_2, l}$. 

We extend the notation $|\cdot|$ as norm of vectors in $\R^2$ into a function in two complex variables, by writing
\[
	|(z_1, z_2)| = \sqrt{z_1^2 + z_2^2}, 
\]
where for the function $\sqrt{\cdot}$, we consider the analytic extension of the real function to the slit plane $\C \setminus (-\infty, 0]$.  Note that $|\cdot|$ can take a complex value and is no longer a norm in $\C^2$. We use $\|r\| = \|(z_1, z_2)\| = \sqrt{z_1\bar{z}_1 + z_2 \bar{z}_2}$ to denote the standard norm in $\C^2$, and note that $\bigl\|\,  |r| \, \bigr\| \le \|r\|$. We use the same notations $\cA, \cD$ for analytic functions into $\C^2$,  where the $\|\cdot\|$ norm is used. 

When $r$ is analytic, we will assume
there is $\sigma_0 >0$, such that
\begin{equation}
	\label{eq:std-assumption}
	\|r\|_{\sigma_0, 3} \le D. \tag{A2}
\end{equation}

For this section, we consider $r$ as fixed and $L$ a function of $(s, s')$. 
\begin{lem}\label{lem:analytic-norm}
Let $r$ satisfy \eqref{eq:real-assumption} and \eqref{eq:std-assumption},: 
\begin{enumerate}
 \item If $\sigma_0 < \frac12$, we have $|\dot{r}| \in \cA_{\sigma_0}$ with $(2D)^{-1} \le \bigl\| |\dot{r}(s)| \bigr\|_{\sigma_0} \le D$ for all $s \in \T_{\sigma_0}$. 
 \item The function $L(s, s + h)/h \in \cD_{\sigma_0/2, \sigma_0/2}$, moreover
 \[
 \min \left\| L(s, s + h)/h \right\| \ge (2D)^{-1}, \quad \left\| L(s, s +h)/h \right\|_{\sigma_0/2, \sigma_0/2, 1} \le D.
 \]
 \item The functions $\partial_1 L(s, s +h), \partial_2 L(s, s + h)\in \cD_{\sigma_0/2, \sigma_0/2}$. Moreover, there exists $C_0, D_0 >1$ depending only on $D$, such that if $\sigma < C_0^{-1}$, for all $(s, h) \in \T_{\sigma_0/2} \times B_{\sigma_0/2}$, 
 \[
 D_0^{-1} \le \left\|\partial^2_{12} L(s, s + h)/h \right\|, \, \|\sqrt{1 - \partial_1 L(s, s + h)}/h\|  \le D_0, 
 \]
 and $\|\partial^3 L(s, s +h)\|_{\sigma_0/2, \sigma_0/2} \le D_0 $, where $\partial^3$ denote arbitrary third partial derivatives of $L$. 
\end{enumerate}
\end{lem}
\begin{proof} Write $\sigma = \sigma_0$ for short. 

(1) Write $\dot{r}(s) = (v_1, v_2)(s)$, then $|(v_1, v_2)|^2 = v_1^2 + v_2^2$ is an analytic function on $\T_\sigma$. Since
\begin{equation}
  \label{eq:r2-bound}
  \left\| \frac{d}{ds}(v_1^2 + v_2^2) \right\| = \|v_1 \dot{v}_1 + v_2 \dot{v}_2\| \le \|\dot{r}\| \|\ddot{r}\| \le D^2, \quad s \in \T_\sigma,  
\end{equation}
we have 
\[
\Re (v_1^2(s) + v_2^2(s))  \ge D^{-2} - \left\| \frac{d}{ds}(v_1^2 + v_2^2) \right\|_\sigma \sigma \ge D^{-2} - D^2 \sigma > (2D)^{-2} . 
\]
Since  $\sqrt{\cdot}$ is extensible to  $\C \setminus {(-\infty, 0]}$, the function $|\dot{r}(s)|$ is extensible to $\T_\sigma$. The lower and upper bound for the norm follows easily from the above estimate. 

(2) Consider the function 
\[
f(s, h) = \begin{cases}
\frac{1}{h}(r(s+h) - r(s)) & h \ne 0 \\
\dot{r}(s) & h = 0, 
\end{cases}
\]
then $f$ is a real analytic function extensible to $\T_\sigma \times B_\sigma$. We have
\[
\|f(s, h)\|_{\sigma/2, \sigma/2} : = \sup_{(s, h) \in \T_{\sigma/2} \times B_{\sigma/2}} \left\| f(s, h) \right\| \le D
\]

Moreover, direct computations yield
\[
\|\partial_sf(s, h)\|_{\sigma/2, \sigma/2} , \|\partial_h f(s, h)\|_{\sigma/2, \sigma/2} \le \|r\|_{\sigma, 2} \le D. 
\]
For $(s, h)\in \T_{\sigma/2} \times B_{\sigma/2}$, denote $s_t = \Re s + t (s - \Re s)$, $h_t = th$, $t \in [0, 1]$,  we get 
\[
\begin{aligned}
& \Re |f(s, h)|^2 > f(\Re s, 0)^2 - \left\| \frac{d}{dt}|f(s_t, h_t)|^2 \right\| 
& > D^{-2} - (2D)^2\sigma > (2D)^{-2}
\end{aligned}
\]
therefore $\sqrt{|f(s, h)|^2}$ is analytic on $\T_{\sigma/2} \times B_{\sigma/2}$. The norm bounds follow from the estimates obtained so far.

(3) Finally, notice 
\[
\partial_1 L(s, s + h) = -  \left\langle  \frac{r(s + h) - r(s)}{|r(s + h) - r(s)|}, \dot{r}(s) \right\rangle 
= -  \left\langle  \frac{(r(s + h) - r(s))/h}{|r(s + h) - r(s)|/h}, \dot{r}(s) \right\rangle 
\]
is in $\cA_{\sigma/2, \sigma/2}$. Similarly for $\partial_2 L$. Direct computation shows $\partial^3 L$ are estimated by up to $3$ derivatives in $r$, hence the norm estimates. 

Finally, in Appendix~\ref{sec:gen-func} we compute the Taylor expansion of $L(s, s+h)$, when $s$ is the arclength parameter:
\[
L(s, s + h) = h\left( 1 - \frac{1}{24} \kappa^2(s) h^2 + O(h^3) \right), 
\]
which implies for real values
\[
\lim_{h \to 0+} \partial^2_{12} L(s, s + h)/h = - \frac14 \kappa(s),  \quad
\lim_{h \to 0+} \sqrt{1 - \partial_1 L(s, s + h)}/h = \frac12\kappa(s).
\]
When $s$ is not the arclength parameter, these limits will become $-\frac14|\dot{r}| \kappa$ and $\frac{1}{2\sqrt{|\dot{r}|}}\kappa$. We conclude that these limits are uniformly bounded away from $0$. Item (3) follows using an argument similar to item (2).  
\end{proof}

\begin{lem}\label{lem:func-def}
There is $C_1, D_1 > 1$ depending only on $D$, such that if  $r$ satisfy  \eqref{eq:real-assumption} and \eqref{eq:std-assumption}, and
\begin{equation}
	\label{eq:init-bound}
	\|u_\theta - 1\|_{\sigma_0} < \frac12, \quad \|u_{\theta\theta}\|_{\sigma_0} < 1,  \quad \alpha = q^{-1} < C_1^{-1} \sigma_0, 
\end{equation}
 then for $\sigma < \sigma_0/4$, $E(r, u)$ (as a function of $\theta$) is analytic on $\T_{\sigma}$. Moreover, for all $\theta \in \T_\sigma$:
 \begin{enumerate}
 \item $\|E(r, u)\| < D_1$;
 \item $D_1^{-1} \alpha \le  \|\partial_{12}L(u, u^+)\| \le D_1 \alpha$;
 \item $\|\partial^3 L(u, u^+)\| \le D_1$. 
\end{enumerate}
\end{lem}
\begin{proof}
The assumption	$\|u_\theta - 1\|_\sigma < \frac12$ implies $u: \T_\sigma \to \T_{2\sigma}$ is well defined. For $\theta \in \T_\sigma$, we have $\|u_\theta(\theta) \alpha  - \alpha\| < \frac\alpha{2}$, and
\[
 \frac14 \alpha < \|u(\theta + \alpha) - u(\theta)\| = \|u_\theta(\theta) \alpha\| + O(\|u\|_{\sigma, 2} \alpha^2) < 2\alpha
\]
As a result, if \eqref{eq:init-bound} holds with sufficiently large $C_1$,  $(u(\theta), u(\theta + \alpha) - u(\theta)) \in \T_{\sigma} \times B_{\sigma}$. The estimates (1)(2)(3) then follows from Lemma~\ref{lem:analytic-norm}. 
\end{proof}



We now discuss the inverse of the operators $\nabla$ and $\nabla^-$. 
\begin{lem}
	Given any function $g \in \cA_\sigma$ satisfying $[g]_\sigma = 0$, there is a unique function $\phi$ with $[\phi]_q =0$,  such that 
	\[
		\nabla \phi = g, \quad \|\phi\|_\sigma \le q \|g\|_\sigma. 
	\]
	The same holds with $\nabla$ replaced with $\nabla^-$. 
\end{lem}
\begin{proof}
	It's easy to see that the function we seek is given by the Fourier series $\phi = \sum \phi_k e^{2\pi i k \theta}$, where
	\[
		\begin{cases}
			\phi_k = g_k/(e^{2 \pi i k/q} - 1), & k \notin q \Z, \\
			\phi_k = 0, & k \in q \Z. 
		\end{cases}
	\]

	To get the norm estimate, we compute  $\varphi$ in a different way. 
	Since $\nabla \phi = \phi(\theta + \alpha) - \phi(\theta) = g(\theta)$, we have, for $j = 1, \cdots, q$ , 
	\[
		\phi(\theta + j \alpha) - \phi(\theta) = j g(\theta + (j-1) \alpha). 
	\]
	Sum over all $j$, we get 
	\[
		\sum_{j = 1}^q \phi(\theta + j \alpha) - q \phi(\theta) = \sum_{j = 1}^q j g(\theta + (j-1)\alpha), 
	\]
	since $\sum_{j = 1}^q \phi(\theta + j \alpha) = q[\phi]_q = 0$, then 
	\[
		\phi(\theta) = - \frac{1}{q} \sum_{j = 1}^q j g(\theta + (j-1)\alpha). 
	\]
	The norm estimate follows easily. 
\end{proof}

\begin{cor}\label{cor:inverse-bound}
Suppose $r$ satisfies condition \eqref{eq:real-assumption} and \eqref{eq:std-assumption}. Then for each $[g]_q = 0$, then there exists unique $[w]_q = 0$ such that 
\[
	\nabla^-\left( L_{12}(r, \id) u_\theta u_\theta^+ \nabla w \right) = g. 
\]
Moreover,  there is $C_2, D_2 >1 $ depending only on $D$ such that if $q^{-1} < C_2^{-1}$, and $\sigma < \sigma_0/4$, we have 
\[
	\|w\|_\sigma \le D_2 q^3 \|g\|_\sigma. 
\]
\end{cor}
\begin{proof}
	Let $p^{-1} = L_{12} u_\theta u_\theta^+$, we attempt to solve 
	\[
		\begin{cases}
			\nabla^- h = g \\
			p^{-1}\nabla w = h + h_1. 
		\end{cases}
	\]
	The first equation can be solved directly. For the second equation, we need to solve 
	\[
		\nabla w = ph + p h_1. 
	\]
	The equation has a unique solution if  $[ph]_q + [ph_1]_q = 0$. Suppose $h_1(\theta + \alpha) = h_1(\theta)$, then $[ph_1]_q = h_1 [p]_q$, 	therefore we can take
	\[
		h_1 = -  [ph]_q / [p]_q . 
	\]

	We now estimate the norm. 	Lemma~\ref{lem:func-def}, item (2) implies
	\[
		C^{-1} q \le \|p\|_\sigma \le C q, 
	\]
	We now have the estimate: 
	\[
		\|h_1\|_{\sigma} \le \frac{\|p\|_{\sigma} \|h\|_{\sigma}}{ \min_{\T_{\sigma}} |p|} \le C \|h\|_\sigma .
	\]
	 As a result, 
	\[
		\|w\|_\sigma \le C q \|p\|_\sigma \|h\|_\sigma \le C q^2 \|h\|_\sigma, \quad 
		\|h\|_\sigma \le q \|g\|_\sigma,\quad
		\|w\|_\sigma \le C q^3 \|g\|_\sigma. 
	\]
\end{proof}

\section{Nekhoroshev-type iteration}
\label{sec:nek}

Recall that  $r_0$ satisfies \eqref{eq:real-assumption} and \eqref{eq:std-assumption}. There is a constant $C_3> 1$ depending only on $D$, such that if 
\begin{equation}
  \label{eq:B-sigma}
  	r \in \cB_\sigma := \left\{  \|r - r_0\|_{\sigma, 3} < C_3^{-1}  \right\}, 
\end{equation}
$r$ satisfies the conditions \eqref{eq:real-assumption} and \eqref{eq:std-assumption} with $D$ replaced by $2D$. 

The goal of this section is to prove the following Nekhoroshev-type result, which allows us to reduce the non-resonant part of $E(r, \id)$ to an exponentially small function. 
\begin{prop}\label{prop:nek}
Suppose $\|r - r_0\|_{\sigma_0, 2} < (2C_3)^{-1}$, then there exists $C_4, D_4 > 1$ depending on $D$, such that if
\[
	\sigma < \sigma_0/4, \quad \|E(r, \id)\|_{\sigma} = : \epsilon  < C_4^{-1} q^{-7}, \quad q^{-1} < C_4^{-1}, 
\]
there exists an analytic diffeomorphism  $u_\nek: \T \to \T$ satisfying $\|u_\nek - \id\|_{\sigma/2}< D_4 q^3 \epsilon$, and for  $r_{\nek} = r \circ u_\nek$, we have  $\|r_{\nek} - r\|_{\sigma/2, 3} < 2 \epsilon$ and 
\[
	\left\| E(r_\nek, \id)\right\|_{\sigma/2} < D_4 \epsilon  e^{- q\sigma/2}.  
\]
\end{prop}

We now describe the iteration process. 
\begin{lem}
	[Iteration lemma for Moser-Levi procedure] \label{lem:nek-iter} Suppose $\|r - r_0\|_{\sigma, 3} < (2C_3)^{-1}$, 
	and let $v$ solve
	\begin{equation}
		\label{eq:v1}
		\nabla^-(L_{12}(r, \id) \nabla v) = - \{E(r, \id)\}_q, 
	\end{equation}
	then there exists constants $C_5, D_5 > 1$ depending only on $D$ such that: if for $\eta, \epsilon > 0$ and $0 < \sigma' < \sigma < \sigma_0/4$,  
	\[
	  \|E(r, \id)\|_\sigma < \eta, \quad  \|\{E(r, \id)\}_q\|_\sigma < \epsilon, \quad C_5 q^3 \epsilon < \sigma - \sigma', \quad q^{-1} < \min\{C_5^{-1}, C_1^{-1} \sigma_0\}, 
	\]
	where $C_1$ is from \eqref{eq:init-bound}, then $r_+ = r \circ (\id + v) \in \cA_{\sigma', 3}$, 
	\[
		\|r_+ - r\|_{\sigma', 3} < \frac{D_5 q^3\epsilon}{(\sigma - \sigma')^3} , \quad
		\|v\|_{\sigma'} < D_5  q^3 \epsilon, 
	\]
	and 
	\[
		\left\| E(r_+, \id) - [E(r, \id)]_q  \right\|_{\sigma'}, \, \,  \left\| \{E(r_+, \id)\}_q  \right\|_{\sigma'} \le  D_5 \left( \frac{  q^3 \eta }{(\sigma - \sigma')} +  q^6 \epsilon \right) \epsilon. 
	\]
\end{lem}
\begin{proof}
By Corollary~\ref{cor:inverse-bound},  $\|v\|_\sigma \le C q^3 \epsilon$. 
	Let $\delta = (\sigma - \sigma')/2$, for sufficiently large $C_5$, we have $\|v\|_\sigma < \delta$ and therefore $u = \id + v: \T_{\sigma'} \to \T_\sigma$ is well defined. Moreover, for sufficiently mall $\epsilon$, the conditions \eqref{eq:init-bound} is satisfied. We have
	\[
		\|r_+ - r\|_{\sigma - \delta} = \|r\circ (\id + v) - r\|_{\sigma - \delta} \le \|r\|_{\sigma} \|v\|_{\sigma - \delta} \le C q^3 \epsilon.
	\]
	and
	\[
		\|r_+ - r\|_{\sigma- 2\delta, 3} = \delta^{-3} \|r_+ - r\|_{\sigma - \delta} \le \frac{Cq^3\epsilon}{(\sigma - \sigma')^3} . 
	\]

By Lemma~\ref{lem:analytic-norm}, item (3),  
	\[
		\left\| \frac{d^2}{dt^2} E(r, \id + tv) \right\|_{\sigma} \le C \|\partial^3 L\|_\sigma \|v\|_\sigma^2 \le C \|v\|_\sigma^2, 
	\]
	and 
	\[
		\begin{aligned}
			& \left\| E(r_+, \id) - [E(r, \id)]_q  \right\|_{\sigma'} \le  \|E(r, \id) - [E(r, \id)]_q + \partial_u E(r, \id) \cdot v\|_{\sigma'} + C  \|v\|_\sigma^2 \\
			& \le \| v \frac{d}{d\theta} E(r, \id)\|_{\sigma'} + C \|v\|_\sigma^2
			\le C \frac{q^3 \eta \epsilon}{\sigma - \sigma'} + C q^6 \epsilon^2. 
		\end{aligned}
	\]
	We get our final estimate by applying the operator $\{\cdot\}_q$ to the above formula. 
\end{proof}

\begin{proof}[Proof of Proposition~\ref{prop:nek}]
We now perform the Nekhoroshev iteration. As before $C>1$ denote a generic constant, but in this proof $C$ may depend on both $D$ and $\sigma_0$. 

We take $\delta = \sigma/(2q)$, define $\epsilon_1 = \epsilon$, $\eta_1 = \epsilon$ and $\sigma_n = \sigma_1 - (n-1)\delta$, define $\sigma_n = \sigma_1 - (n-1)\delta$, let 
$v_n$ solve \eqref{eq:v1} for $r = r_{n-1}$,  define  $r_n = r_{n-1} \circ (\id + v_n)$.  Then we have 
\[
\begin{aligned}
 	\epsilon_2 & = \left\| E(r_2, \id) - [E(r_1, \id)]_q \right\|_{\sigma_2} < D_4 \left( \frac{q^3 \eta_1}{\delta_*}+ q^6 \epsilon_1 \right) \epsilon_1  \\
  &\le  (2D_4 q^6 \epsilon_1)\epsilon < \frac{\epsilon_1}2  
\end{aligned}
\]
since $q^{-1} < C \sigma$ and $\epsilon < C q^{-7}$ implies $q^3\epsilon/\delta = q^4\epsilon_1/\sigma < q^6 \epsilon_1$, and $2D_4 q^6 \epsilon_1 < \frac12$. Moreover,
\[
	\|r_2 - r_1\|_{\sigma_2, 3} < \frac{D_4 q^3 \epsilon_1}{\delta^3} < D_4 q^6 \epsilon < D_4 q^{-2} < (4C_3)^{-1}
\]
as long as $q^{-1} < C_4^{-1}$ is small enough.

 We now check that as long as $\sigma_n = \sigma_1 - n \delta > 0$, the following holds inductively:
\begin{enumerate}
\item $\epsilon_n = \left\| E(r_n, \id) - [E(r_{n-1}, \id)]_q \right\|_{\sigma_n} < 2^{-(n-1)}\epsilon_1$. 
\item For $n \ge 2$, $\|r_n - r_{n-1}\|_{\sigma_n} < 2^{-(n-1)} (4C_3)^{-1}$, and $\|r_n - r_1\|_{\sigma_n} < (2C_3)^{-1}$, in particular, $r_n \in B_{\sigma_n}$.  
\item $\eta_n = \|E(r_n, \id)\|_{\sigma_n} \le \epsilon_n + \eta_{n-1} < 2 \epsilon$. 
\item $u_n = (\id + v_n) \circ \cdots \circ (\id + v_1)$ satisfies $\|u_n - \id\|_{\sigma_n} < C q^3\epsilon$. 
\end{enumerate}
Pick $N = q$, we get $\sigma_N = \sigma - q\delta = \sigma/2$, the above estimates implies $\|E(r_N, \id)\|_{\sigma/2} < 2\epsilon$, and
\[
	  \|\{E(r_N, \id)\}_q\|_{\sigma/2} = \epsilon_N < 2^{-q} \epsilon = \epsilon e^{-(\log 2) q}. 
\]
On the other hand, using $\|E(r_N, \id)\|_{\sigma/2} \le 2\epsilon$, and by applying a standard estimate (Lemma~\ref{lem:res-exp}), we have 
\[
	\left\| [E(r_N, \id)]_q \right\|_{\sigma/4} \le 2\epsilon e^{-\frac14 \pi q \sigma} < 2\epsilon e^{-\sigma q/2}.     
\]
For small $\sigma$ the second upper bound is larger. The proposition follows by taking $u_\nek = u_N$, $r_\nek = r_N$, and taking the sum of the two upper bounds obtained. 
\end{proof}

\begin{lem}\label{lem:res-exp}
Let $f\in \cA_{\sigma}$ satisfy  $[f]_q = f$ , $\int_0^1 f(\theta) d\theta = 0$  and $q^{-1} < C \sigma$, then there is an explicit $C' > 0$ depending only on $C$ such that 
\[
	\|f\|_{\sigma/2} < C' e^{-\frac12 \pi q\sigma} \|f\|_\sigma. 
\]
\end{lem}
\begin{proof}
Our conditions implies $f(\theta) = \sum_{k \in \Z \setminus\{0\}} f_{kq} e^{2\pi i k q \theta}$. Denote $\epsilon = \|f\|_\sigma$, standard estimates of the analytic function implies $|f_k| \le \epsilon e^{- 2\pi k q \sigma}$, then 
\[
	\|f\|_{\sigma/2} \le  |\sum_{k \in \Z\setminus \{0\} } f_k e^{2\pi q k\sigma/2}| \le \epsilon e^{-2\pi q\sigma/4} \sum_{k \in \Z \setminus \{0\}}  e^{-2\pi kq\sigma/4} \le C' \epsilon e^{-2\pi q\sigma/4}, 
\]
where  we used $\sigma q > C^{-1}$, and set $C' = \frac{1}{1 - e^{\frac12 \pi C^{-1}}}$. 
 \end{proof}

\section{Lazutkin-type normal form and smooth approximation}
\label{sec:laz}

In this section we perform an initial step of normal form due to Lazutkin (\cite{Laz1973}). For this, we use the billiard map for the first time. In this section, $s$ will be the arclength parameter, namely $|\dot{r}(s)| = 1$ for all $s \in \T$. Let $\vartheta$ denote the angle between the outgoing billiard ray with the positive tangent vector at the impact point.  Near the boundary (see for example \cite{Laz1973}), the billiard map  can be written approximately as:
\[
	T(s, \vartheta) = \bmat{ s + 2 \rho(s) \vartheta + O(\vartheta^2) \\ \vartheta - 2 \rho'(s) \vartheta^2/3 + O(\vartheta^3)},
\]
where $\rho(s)$ is the radius of curvature at $r(s)$.

\begin{lem}[See also \cite{MRT2016}, Proposition 5]\label{lem:billiard-analytic}
There exists a constant $C_5>1$ depending only on $D$ such that: if $r$ satisfies conditions \eqref{eq:real-assumption} and \eqref{eq:std-assumption}, and assume $\sigma_0 < C_5^{-1}$. Then the map $T$ is real analytic, and can be extended to a complex analytic map on  $\T_{\sigma_0/2} \times B_{\sigma_0/2}$. Moreover, for each $\sigma < \sigma_0$
\[
	\|T\|_{\sigma} \le C_5. 
\]
\end{lem}
\begin{proof}
Suppose $T(s, \vartheta) = (s^+, \vartheta^+)$, and write $h =s^+ - h$, then
\[
	\cos \vartheta = \partial_1 L(s, s + h), \quad  - \cos \vartheta^+ = \partial_2 L(s, s + h). 
\]
For real $h > 0$, the above equations are equivalent to (see also Appendix~\ref{sec:gen-func})
\[
2\sin(\vartheta/2) = \sqrt{2(1 + \partial_1 L(s, s + h))}, \quad 
2\sin(\vartheta^+/2) = \sqrt{2(1 - \partial_2 L(s, s + h))}. 
\]
According to Lemma~\ref{lem:analytic-norm}, the derivative 
\[
\partial_h \sqrt{2(1 + \partial_1 L(s, s + h))} = \frac{\partial_{12}L(s, s + h)}{\sqrt{2(1 + \partial_1 L(s, s + h))}} = \frac{\partial_{12}L(s, s + h)/h}{\sqrt{2(1 + \partial_1 L(s, s + h))}/h}
\]
is analytic on $\T_{\sigma_0/2} \times B_{\sigma_0/2}$ uniformly bounded away from $0$. As a result, the implicit function applies, which also come with an estimate on the norm. 
\end{proof}

\begin{prop}\label{prop:gen-laz}
Under the same assumption as Lemma~\ref{lem:billiard-analytic}, for each $k \ge 0$, there exists a real analytic coordinate change $\Phi = \Phi_{2k + 4}: (s, \vartheta) \mapsto (x, y)$, such that $T_{\laz} = \Phi \circ T \circ \Phi^{-1}$ takes the form:
\[
	T_{\laz}(x, y) = (x + y + y^{2k+4} f(x, y), y + y^{2k+4} g(x, y)). 
\]
Moreover, there is a constant $C_6>1$ depending only on $D$ and $k$ such that, for every $\sigma < \sigma_0/2$, we have 
\[
	\|f\|_{\sigma}, \|g\|_{\sigma} \le C_6 (\sigma_0/2 - \sigma)^{-2k-3} . 
\]
\end{prop}
The proof is presented in the Appendix.

\begin{cor}\label{cor:appox-general}
Under the same assumptions as Proposition~\ref{prop:gen-laz}, there is $C_7 > 1$ such that, for $\sigma_1 \le \sigma_0/4$,  there is a real analytic diffeomorphism  $u_q: \T \to \T$ extensible to $\T_{\sigma_1}$, such that 
\[
	\|E(r, u_q)\|_{\sigma_1} < C_7 \sigma_0^{-7} q^{-8}. 
\]
\end{cor}
\begin{proof}
Denote $\alpha = 1/q$ as before, and let $\Phi = \Phi_8$ be the Lazutkin normal form of order $8$ from Proposition~\ref{prop:gen-laz}. 
Let $\gamma(\theta) = (\theta, \frac{1}{q}) \in \T \times (0, \sigma_0/2)$. For 
\[
	\eta(\theta) = (u(\theta), w(\theta)) = \Phi^{-1}(\gamma(\theta)),  
\]
we have $\|T(\eta(\theta)) - \eta(\theta+\alpha)\|_{\sigma_1} \le C \sigma_0^{-7} q^{-8}$. 

Let $\olu(\theta) = \pi_1 T(u(\theta), w(\theta))$, and $\ulu(\theta) = \pi_1 T^{-1}(u(\theta), w(\theta))$, since $\olu(\theta), u(\theta), \ulu(\theta)$ is the projection of a billiard orbit, we have
\[
	\partial_{s'} L(r; \ulu(\theta), u(\theta)) + \partial_s L(r, u(\theta), \ulu(\theta)) = \cos w(\theta) - \cos w(\theta) = 0.  
\]
Moreover,
\[
	\|\olu(x) - u(x + \alpha)\|_{\sigma_1}, \|\ulu(x) - u(x -\alpha)\|_{\sigma_1} \le C \sigma_0^{-7}q^{-8}, 
\]
we conclude that  
\[
	\|E(r, u)\|_{\sigma_1} \le C \sigma_0^{-7} q^{-8}. 
\]
\end{proof}

Apply Corollary~\ref{cor:appox-general} and then Proposition~\ref{prop:nek}, we obtain:
\begin{cor}
[Analytic case] \label{cor:app-analytic}
Suppose $r$ satisfies the conditions \eqref{eq:real-assumption} and \eqref{eq:std-assumption}, then there is $C_8 > 1$ depending only on $\sigma_0$ and $D$ such that for $q^{-1} < C_8^{-1}$ and $\sigma_1 = \sigma_0/8$, there exists $u_\app: \T \to \T$ extensible to $\T_{\sigma_1}$ such that
\[
	\|E(r \circ u_\app, \id)\|_{\sigma_1} <  C_8 q^{-8} e^{- \sigma q / 8 }. 
\] 
\end{cor}
\begin{proof}
We first apply Corollary~\ref{cor:app-analytic}, choosing the width $\sigma_0/4$. Then Proposition~\ref{prop:nek} applies with $\epsilon = C \sigma_0^{-7} q^{-8}$ and parameter $\sigma_0/4$. We obtain $u_\nek: \T_{\sigma_0/8} \to \T_{\sigma_0/4}$, such that $\|E(r \circ u_q \circ u_\nek, \id)\|_{\sigma_1} < C \epsilon e^{- \sigma_0 q/ 8}$ 
\end{proof}

In the case $r$ is smooth, we use a standard analytic approximation. 
\begin{prop}
[See \cite{Zeh1975}, Lemma 2.1, 2.3]
Suppose $f: \R^n \to \R$ is $C^m$, and let $l < m$. Then for each $t> 0$ there exists an analytic function $S_t f$ satisfying the following estimates:
\[
	\|S_t f\|_{t^{-1}} \le C_{n, m} \|f\|_{C^m}, \quad
	\|S_t f - f\|_{C^l} \le C_{n, m, l} t^{-(m - l)} \|f\|_{C^m}. 
\]
\end{prop}

For each $r \in C^m$, first assume it is in the arclength parameter. We consider the curvature function $\kappa(s) = |\ddot{r}|$ and the analytic approximation $S_t \dot\kappa$ of $\dot\kappa$,  such that
\[
	\|S_t \dot\kappa\|_{t^{-1}} \le C_m \|\dot{\kappa}\|_{C^{m-3}}, \quad \|S_t \dot\kappa - \dot\kappa\|_{C^{l-3}} \le C_{m, l} t^{-(m - l)} \|\dot\kappa\|_{C^{m-3}}. 
\]
Let us note the curvature function $\kappa$ determines a billiard boundary if and only if (see \cite{MM1982}, Proposition 2.7) 
\[
\kappa > 0, \quad \int_0^1 \kappa(s) cos(2\pi s) ds = \int_0^1 \kappa(s) \sin(2\pi s) ds = 0.
\]
Consider $\kappa_1(s) = \kappa(0) + \int_0^s \left( S_t\dot{\kappa}(\tau) - \int_0^1 S_t \dot{\kappa}(\sigma)d\sigma \right) d\tau$, and $\kappa_\omega$ be the projection in Fourier space of  $\kappa_1$ by removing the modes $\cos(2\pi s)$ and $\sin(2\pi s)$. Then there exists a billiard boundary $r_\omega$ corresponding to the curvature $\kappa_\omega$. 

We will pick $t = q^{\frac17}$, $\sigma_q = q^{-\frac17}$. Then 
\[
	\|r_\omega\|_{\sigma_q, 3} \le C_m \|r\|_{C^{m}}, \quad \|r_\omega - r\|_{C^{l+3}} \le C_{m, l} q^{-\frac{m-l}{7}} \|r\|_{C^{m}}. 
\]
We apply Corollary~\ref{cor:appox-general} to $r_\omega$, and obtain an approximate solution $u_q$ such that $\|E(r_\omega, u_q)\|_{\sigma_q/2} < C \sigma_q^{-7} q^{-8} < C q^{-7}$, which satisfies the condition of Proposition~\ref{prop:nek}. After application of the proposition, we obtain: 
\begin{cor} 
[Smooth case] \label{cor:app-smooth}
Suppose $r \in C^m$ satisfies \eqref{eq:real-assumption} and $\|r\|_{C^m} < D$. Then there is $C_9> 1$ depending only on $D$ and $m, l$ such that, for 
\[
	q^{-1} < C_9^{-1}, \quad \sigma_q = q^{-\frac17}, 
\]
there is $r_\omega \in \cA_{\sigma_q, 2}$ and $u_\app \in \cA_q$ satisfying $\|r_\omega\|_{\sigma_q, 3} \le C_9$, and 
\[
	\|r_\omega - r\|_{C^l} \le C_9 q^{-\frac{m-l}{7}}, \quad
	\|E(r_\omega \circ u_\app, \id)\|_{\sigma_q/4} \le C_9 e^{-  C_9^{-1} q^{\frac78}}. 
\]
\end{cor}

\section{KAM algorithm and proof of the main theorem}
\label{sec:kam}

We prove the following KAM-type theorem. 
\begin{thm}\label{thm:analytic-KAM}
Suppose $0 < \sigma < \sigma_0/2$, and $r_0$ satisfies our standing assumptions \eqref{eq:real-assumption} and  \eqref{eq:std-assumption}. There exists constants $C_{10}, D_{10} > 1$ depending only on $D$ such that if 
\[
	q^{-1} < C_{10}^{-1}, \quad  \|E(r, \id)\|_\sigma :=  \epsilon < C_{10}^{-1} q^{-8}, \quad  \epsilon <  C_{10}^{-1} q^{-5} \sigma^4, 
\] 
 there is $r_\infty \in \cA_{\sigma/2}$ such that 
\[
	E(r_\infty, \id) = 0, \quad \|r_\infty - r\|_{\sigma/2, 3} < D_{10} q^3\epsilon/\sigma^3. 
\]
\end{thm}

We now describe the iteration process. 
\begin{prop}
	[Iterative lemma for KAM] \label{prop:iter} Suppose $\|r - r_0\|_{\sigma, 2} < (2C_3)^{-1}$ ($C_3$ is from \eqref{eq:B-sigma}), define 
	\begin{equation}
		\label{eq:a-iter} 
		a(\theta) = - \int_0^\theta \frac{[E(r, \id)]_q}{[F(r, \id)]_q} (\tau) d\tau, \quad r_* = e^a r,
	\end{equation}
	and let $v$ solve
	\begin{equation}
		\label{eq:v}
		\nabla^-(L_{12}(r_*, \id) \nabla v) = - E(r_*, \id), 
	\end{equation}
	then there exists constants $C_4, D_4 > 1$ depending only on $D$ such that: if for $\epsilon > 0$ and $0 < \sigma' < \sigma$,  
	\[
	 \|E(r, \id)\|_\sigma < \epsilon, \quad C_4 q^3 \epsilon < \sigma - \sigma', \quad
	 q^{-1} < \min\{C_4^{-1}, C_1\sigma_0\}, 
	\]
	where the constants $C_1, C_4$ are from Lemma~\ref{lem:nek-iter}. 
	we have $r_+ = r_* \circ (\id + v) \in \cA_{\sigma', 3}$, and 
	\[
		\|r_+ - r\|_{\sigma', 2} < \frac{D_4 q^3\epsilon}{(\sigma - \sigma')^3} , \quad
		\|v\|_{\sigma'} < D_4  q^3 \epsilon, \quad
		\|E(r_+, \id)\|_{\sigma'} \le  \left( D \frac{q^3}{\sigma - \sigma'} + q^6 \right) \epsilon^2. 
	\]
\end{prop}

\begin{proof}
	Note that for small $\sigma$, Lemma~\ref{lem:analytic-norm} implies $C^{-1}q^{-1} \le \||\Delta r|\|_\sigma \le C q^{-1}$, and therefore by Lemma~\ref{lem:F-const}, 
	\[
		C^{-1} q^{-1} \le \|[F(r, \id)]_q\|_\sigma = \|[|\Delta r|]_q\|_\sigma \le C q^{-1}, 
	\]
	then noting $a$ is $\alpha = q^{-1}$ periodic, 
	\[
		\|\dot{a}\|_{\sigma} \le \frac{\|E\|_\sigma}{\min \|F\|} \le C q \epsilon, \quad
		\|a\|_\sigma \le  \max_{0 \le t \le \alpha, \,  \theta \in \T_\sigma} \left|\int_0^t \dot{a}(\theta + \tau) d\tau \right| \le C \alpha q \epsilon = C \epsilon. 
	\]
	As a result
	\[
		\|r_* - r\|_{\sigma', 3} \le C \|a\|_{\sigma,3} \le \frac{C\epsilon}{(\sigma - \sigma')^3}. 
	\]
	Moreover, 
	\[
		\|E(r_*, \id)\|_{\sigma} \le \|e^a\|_\sigma (\|\dot{a}\|_\sigma \|F(r, u)\|_\sigma + \|E(r, u)\|_\sigma) \le  C q\epsilon \cdot  \alpha + C\epsilon \le C \epsilon. 
	\]
	We now apply Lemma~\ref{lem:nek-iter} to $r_*$, to obtain the estimate for $r_+$. Note that the bound $\|r_+ - r_*\|$ dominates the bound $\|r_* - r\|$, therefore the final error estimates are of the same order as in Lemma~\ref{lem:nek-iter}. 
\end{proof}

To apply KAM we have the following standard induction lemma: 
\begin{prop}
Suppose $0 < \sigma < \sigma_0/2$, and $\|r - r_0\| < (4D)^{-1}$. Denote, 
\[
	\sigma_1 = \sigma, \quad \sigma_{n+1} = \sigma_n - \delta_n, \quad \delta_n = 2^{-n-1}\sigma, 
\]
\[
	 \epsilon_1 = \epsilon, \quad \epsilon_{n+1} = (\epsilon_n)^{\frac54}, \quad n \ge 1. 
\]
There exist constants $C_{10}, D_{10} > 1$ depending only on $D$, such that if:
\[
 C_{10} \epsilon < q^{-8}, \quad C_{10} \epsilon < q^{-4} \sigma^4, 
\]
the following hold for all $n \ge 1$:
\begin{enumerate}
 \item $C_{10} \epsilon_n < q^{-8}$, $C_{10} \epsilon_n < q^{-4} (\delta_n)^4$. 
 \item For $r_1 = r$ and $r_{n+1} = (r_n)_+$ using Proposition~\ref{prop:iter}, we have 
 	\[
		\|r_n- r_{n-1}\|_{\sigma_n, 3} < D_{10}\frac{q^3 \epsilon_n}{\delta_n^3}, \quad \|E(r_n, \id)\|_{\sigma_n} < \epsilon_n,
	\] 
	\[
		C_3 q^3 \epsilon_n < \delta_n.
	\]
\item $r_n \to r_\infty$ in $\cA_{\sigma/2}$, with
	\[
		E(r_\infty, \id) = 0, \quad \|r_\infty - r_0\|_{\sigma/2, 2} < D_{10} q^3 \epsilon /\sigma^3. 
	\]  
\end{enumerate}
\end{prop}
\begin{proof}
Since the sequence $\epsilon_n/(\delta_n)^3$ is decreasing if $\epsilon$ is small enough, item (1) is obvious. The only non-trivial estimate in (2) is the estimate of $\|E(r_n, \id)\|_{\sigma_n}$. Suppose item (2) hold up to index $n-1$.  We use Proposition~\ref{prop:iter} to get 
\[
	\|E(r_n, \id)\|_{\sigma_n} \le ( \epsilon_{n-1})^{\frac54} \cdot C_4 (\epsilon_{n-1})^{\frac34} \left( 2q^3 (\delta_{n-1})^{-3} + q^6  \right) 
\]
where the second group (after $\cdot$) in the product is smaller than $2 C_5^{-1}$ by item (1). Choosing $C_5$ large enough yields the desired estimate. Item (3) follows from item (2) by observing that $\sum_{n \ge 1}\epsilon_n/\sigma_n^3 < C \epsilon/\sigma^3$, and $\sigma_n \to \sigma/2$.  
\end{proof}

Theorem~\ref{thm:analytic-KAM} follows directly from what we just proved.

\begin{proof}[Proof of Theorem~\ref{thm:main}]
We now prove our main theorem.

\emph{Case 1, the analytic case}: Suppose $r$ satisfies \eqref{eq:real-assumption} and \eqref{eq:std-assumption}. For $q$ sufficiently large depending only on $D$ and $\sigma_0$, Corollary~\ref{cor:app-analytic} applies, and 
\[
	\|E(r \circ u_\app, \id)\|_{\sigma_0/4} < \epsilon: = C q^{-8} e^{-\sigma_0 q/8}. 
\]
Clearly the assumptions of Theorem~\ref{thm:analytic-KAM} applies, we obtain $r_\infty \in \cA_{\sigma_0/8, 2}$ such that $E(r_\infty, \id) = 0$, and $\|r_\infty - r \circ u_\app\|_{{\sigma_0}/8} < C q^{-5} e^{-\sigma_0 q/8}$. Then the boundary $r_\infty \circ (u_\app)^{-1}$ satisfies our conclusion. 

\emph{Case 2, the smooth case}. Assume $\|r\|_{C^m} \le C$.  Apply Corollary~\ref{cor:app-smooth}, we obtain $\|r_\omega - r\|_{C^l} \le C q^{-\frac{m-l}4}$, such that $\|E(r_\omega \circ u_\app, \id)\|_{\sigma_q/4} \le C e^{-C^{-1} q^{\frac78}}$. For large $q$ depending only on $D$ and $\sigma_0$, Theorem~\ref{thm:analytic-KAM} applies, and we obtain a perturbed boundary for which $E(\rho_\omega, \id) = 0$. Note that for $q$ large enough, $\|r_\infty \circ (u_\app)^{-1} - r_\omega\|$ is bounded by $q^{-\frac{m-l}7}$, since the former is exponentially small. Then $r_\infty \circ (u_\app)^{-1}$ is the boundary we seek. 
\end{proof}

\appendix

\section{Asymptotic expansion of the generating function}
\label{sec:gen-func}

In this section we compute the Taylor expansion of the generating function $L(s, s + h)$ in $h$, which will prove the analyticity of the billiard map, as well as helping with the Lazutkin coordinates. We assume $s$ is the arclength parameter, namely: $r: \R \to \R^2$ satisfy $|\dot{r}| = 1$. Then $r(s) = (\cos \alpha(s), \sin \alpha(s))$ for $\alpha \in [0, 2\pi]$, and $\dot{\alpha}(s) = \kappa(s)$ which is the curvature. For a fixed $s \in \R$, by rotating the axis if necessary, we may assume $\alpha(s) = 0$. Then explicitly:
\begin{equation}
  \label{eq:L-explicit}
  L(s, s + h)  = \left( \left( \int_0^h \cos(\alpha + \tau) d\tau \right)^2 + \left( \int_0^h \sin (\alpha + \tau) d\tau\right) \right)^{\frac12}. 
\end{equation}
Using the asymptotic expansions
\[
 \alpha(s + h) =    \kappa(s) h + \frac12 \dot\kappa(s) h^2 + O(h^3), 
\]
\[
\cos (\alpha(s + h)) = 1 - \frac12 \kappa^2 h^2 + O(h^3),\quad
\sin(\alpha(s + h)) = \kappa h + O(h^3), 
\]
we get 
\[
\begin{aligned}
 L(s, s +h) & = h\sqrt{ (1 - \frac16 \kappa^2 h + O(h^3))^2 + (\frac12 \kappa h + O(h^3))^2 } \\
 & = h \left(  1 - \frac{1}{24} \kappa^2 h^2 + O(h^3) \right). 
\end{aligned}
\]

Let us consider the billiard map $T(s, \vartheta) = (s^+, \vartheta^+)$, then 
\[
\cos \vartheta = - \partial_1 L(s, s^+) = - (1 - \frac18 \kappa^2 h^2 + O(h^3)), \quad \cos \vartheta^+ = \partial_2 L(s, s^+) = 1 - \frac{1}{8}\kappa^2 + O(h^3). 
\]
Denote $h = s^+ - s$, then $2 \sin^2 (\vartheta/2) = 1 - \cos \vartheta = \frac18 \kappa^2 h^2 + O(h^3)$, and therefore the equation 
\begin{equation}
  \label{eq:theta}
  2\sin(\vartheta/2) = \sqrt{2(1 + \partial_1 L(s, s + h))} = \frac12 \kappa h ( 1 + O(h))
\end{equation}
has a unique solution $h(s, \vartheta)$ near $\vartheta = 0$. Moreover, the implicit function theorem apply on the complex neighborhood. Therefore the function $s^+ = s + h(s, \vartheta)$ is analytic. Moreover, the function $\vartheta^+$ can be obtained from the relation $2 \sin (\vartheta^+/2) = \sqrt{2(1 - \partial_2 L(s, s + h))}$ which is also analytic. We note that $\partial^2_{12}L(s, s +h) = \frac14 \kappa^2 h + O(h^2)$, which is needed in the proof of Lemma~\ref{lem:analytic-norm}. 

We now perform a more detailed analysis of the expansion of $L(s, s +h)$ in $h$. Let's call the following expression a differential monomial in $\kappa$:
\[
	P(\kappa)(s)  = \prod_{k = 0}^m \left( \kappa^{(k)}(s) \right)^{a_k}, \quad a_k \in \N, \, k \ge 1,  \quad a_0 \in \R. 
\]
The degree of $P$ is $a_1 + \cdots a_k$. A differential polynomial of degree $k$ is sum of monomials of degree \emph{up to} $k$, and we denote the linear space of such expressions $\bP_k(\kappa)$. Denote by $\bQ(\kappa)[h]$ the power series of the type $\sum_{k = 0}^\infty P_k(\kappa)(s) h^k$, where each $P_k \in \bP_k(\kappa)$. We will also need notation for power series of the type $\sum_{k = 0}^\infty P_{k+i} h^k$, or $\sum_{k = 0}^\infty P_k h^{k+i}$, where $i \in \N$. Denote
by $\bQ_{\ge k}(\kappa)[h]$  the power series in $\bQ(\kappa)[h]$ that contains only terms higher than or equal to $h^k$. Then
\[
\sum_{k = 0}^\infty P_{k+i} h^k \in h^{-i}\bQ_{\ge i}(\kappa)[h], \quad
\sum_{k = 0}^\infty P_k h^{k+i} \in h^i \bQ(\kappa)[h]. 
\]

Observe that if $f(s, h) \in \bQ(\kappa)[k]$, then 
\[
\partial_s f(s, h), \partial_hf(s, h) = \sum_{k = 0}^\infty P_{k+1}(\kappa) h^k \in h^{-1} \bQ_{\ge 1}(\kappa)[h],
\]
and $f(s, h) \in h^{-k} \bQ_{\ge k}(\kappa)[h]$ implies $f(s, 0) \in \bP_k(\kappa)$. 

  Using \eqref{eq:L-explicit}, we get:
\[
L(s, s + h) = h\left( 1 - \frac{1}{24} \kappa^2 h^2 + h^2 \bQ_{\ge 1}(\kappa)[h] \right)
\]
where we abused notation by using $\bQ$ to denote a unspecified function in the same space. Similarly, \eqref{eq:theta} becomes
\[
2 \sin (\vartheta/2) = \frac12 \kappa h \left( 1 + \bQ_{\ge 1}(\kappa)[h] \right).
\]
Using the Lagrange inversion formula, the solution is given by $h(s,\vartheta) = \sum_{n \ge 1} g_n(s) z^n$, where $z = 2 \sin(\vartheta/2)$, and 
\[
 g_n(s) = \frac{d^{n-1}}{dh^{n-1}}\Bigr|_{h = 0} \left( \frac{1}{(\kappa/2)(1 + \bQ_{\ge 1}(\kappa)[h])} \right)^n \in \bP_{n-1}(\kappa). 
 \] 
Using $g_1 = 2\kappa^{-1}$ and $z$ as a power series $\vartheta + O(\vartheta^3)$,  we get 
\[
h(s, \vartheta) = 2\kappa^{-1} \vartheta + \vartheta \bQ_{\ge 1}(\rho)[\vartheta]. 
\]
 $\vartheta^+$  is obtained from $h$ by 
\[
2 \sin(\vartheta^+/2) = \frac12 \kappa h(1 + \bQ_{k \ge 1}(\kappa)[h]) = \vartheta + \vartheta \bQ_{\ge 1}(\kappa)[\vartheta].
\]
Finally, it's more convenient to use the radius of curvature $\rho = \kappa^{-1}$ to generate the polynomial, to avoid negative powers. Note that the change $\kappa = \rho^{-1}$ takes a function in $\bP_k(\kappa)$ to $\bP_{k}(\rho)$. We summarize the discussions so far in the following lemma. 
\begin{lem}\label{lem:map-expand}
The billiard map admits the following asymptotic expansion at $\vartheta = 0$:
\[
T(s, \vartheta) = \left( s + 2\rho \vartheta + \vartheta \bQ_{\ge 1}(\rho)[\vartheta], \,  \vartheta + \vartheta \bQ_{\ge 1}(\rho)[\vartheta] \right). 
\]
\end{lem}

\section{Higher order Lazutkin normal form}

In this section we build on the asymptotic expansion of last section and obtain higher order normal forms of the type
\[
	(x, y) \mapsto (x + y + O(y^m), \quad y + O(y^{m+1})). 
\] 
 These normal forms are already known to Lazutkin (\cite{Laz1973}) in the smooth case, and an analytic version is given in \cite{MRT2016}. We provide a version with explicit estimates on the width of analyticity, inspired by the unpublished notes of J. De Simoi and A. Sorrentino (\cite{DeSimoi}, \cite{Sorrentino}). 

Suppose the billiard map  $T(s, \vartheta) = (s^+, \vartheta^+)$  and its inverse $T^{-1}(s, \vartheta) =(s^-, \vartheta^-)$ admit the expansion 
\[
	s^\pm  = s  + \sum_{k \ge 1} b_k^\pm \vartheta^k, \quad \vartheta^\pm = \vartheta + \sum_{k \ge 2} d_k^\pm \vartheta^k.
\]
Due to the time reversibility of the map, namely if $I(s, \vartheta) = (s, -\vartheta)$, then $I \circ T \circ I = T$, there exists functions $b_k(s), d_k(s)$ such that $b_{2i +1}^\pm = \pm b_{2i + 1}$, $b_{2i}^\pm = b_{2i}$ and $d_{2i}^\pm = \pm d_{2i}$, $d_{2i+1}^\pm = d_{2i}$. Moreover, Lemma~\ref{lem:map-expand} imply $b_k(s), d_k(s) \in \bP_{k-1}(\rho)$. 
For example, 
\[
	b_1 = 2\rho, \quad b_2 = \frac43 \rho \dot{\rho}, \quad b_3 = \frac49 \rho \dot{\rho}^2 + \frac23 \rho^2 \ddot{\rho}, \quad d_2 = -\frac23 \dot{\rho}, \quad
	\text{etc.}
\]

To obtain a normal form, we consider the formal coordinate change 
\begin{equation}
  \label{eq:higher-lazutkin}
  X(s, \vartheta) = \sum_{i = 0}^\infty F_{2i}(s) \vartheta^{2i}, \quad Y(s, \vartheta) = X(s, \vartheta) - X(s^-(s, \vartheta), \vartheta^-(s, \vartheta)). 
\end{equation}
Note that 
\begin{equation}
  \label{eq:Phi-y}
  Y(s, \vartheta) = F(s) - F(s^-) + O(\vartheta^2) = F'(s)\vartheta + O(\vartheta^2). 
\end{equation}
We attempt to solve the formal equation $Y^+ - Y = 0$, namely
\begin{equation}
  \label{eq:homological}
  X(s^+, \vartheta^+) - 2 X(s, \vartheta) + X(s^-, \vartheta^-) = 0. 
\end{equation}
The main result of this section is:
\begin{prop}\label{prop:formal-conj}
There exists a sequence of analytic functions $F_{2i}$, $i \ge 0$, given by the integral formula
\[
\rho^{-\frac{2i}3} F_{2i} = C \int_0^s \rho^{-\frac23}(\tau) d\tau +  \int_0^s \rho^{-\frac23}(\tau)\int_0^\tau \rho^{-\frac43}(\sigma) P(\sigma) d\sigma d\tau, 
\]
where $P \in \bP_{2i + 1}(\rho)(s)$, and  $C$ is determined by periodicity. Then the coordinate change \eqref{eq:higher-lazutkin} using $F_{2i}$ formally conjugate the map $T$ to $(X, Y) \mapsto (X + Y, Y)$ as a power series of $Y$. 
\end{prop}

\begin{proof}
It suffices to solve the formal equation \eqref{eq:homological} up to all orders of $\vartheta$, in view of \eqref{eq:Phi-y}. 

Due to the time reversal symmetry, $F_{2i}(s^+) - 2F_{2i}(s) + F_{2i}(s^-)$ is an even series and therefore vanish of order $2$ at $\vartheta = 0$. This means the leading term of \eqref{eq:homological} is at order $\vartheta^2$ determined only by $F_0$. An explicit computation shows:
\[
\begin{aligned}
 & F_0(s^+) - 2F_0(s) + F_0(s^+)\\
 &   = F'(s^+ + s^- - 2s) + \frac12 F''_0(s) \left( (s^+ - s)^2 - (s^- - s)^2 \right) + O(\vartheta^4) \\
 &   = (2b_2^2)F_0'(s) \vartheta^2 + \frac12 F''_0(s) (2b_1) \vartheta^2 + O(\vartheta^4). 
\end{aligned}
\]

Attempting to eliminate the $\vartheta^2$ term leads to the equation 
\[
	2b_2F'_0(s) + b_1^2 F''_0(s) = 0, \quad \frac43\left(  2 \rho \dot{\rho} F'_0 + 3\rho^2 F''_0 \right) =0, 
\]
whose solution is
\begin{equation}
  \label{eq:F-laz}
  F_0(s) = C \int_0^s \rho^{-\frac23}(z)dz, 
\end{equation}
where $C$ can be chosen to preserve periodicity. This is identical to Lazutkin's choice. 

Let us omit the notation $(\rho)$ from $\bQ(\rho)[h]$. Note that the function $F_0(s + h) - F_0(h) = C \rho^{-\frac23}(s) h + O(h^2)$ is contained in the space $h\bQ[h]$, due to our particular choice of $F_0$. Noting $s^+ - s, s^- - s \in \vartheta \bQ[\vartheta]$, by substituting we obtain $F_0(s^+) - 2F_0(s) + F_0(s^-) \in \vartheta \bQ[\vartheta]$. Moreover, we already eliminated the terms lower than $\vartheta^4$, which means
\[
F_0(s^+) - 2F_0(s) + F_0(s^-) \in \vartheta \bQ_{\ge 3}[\vartheta].
\]

We now proceed by induction. Suppose for some $k \ge 1$, 
\begin{equation}
  \label{eq:inductive}
  R_{2k + 4} : = \sum_{i = 0}^{k} \left( F_{2i}(s^+)(\vartheta^+)^{2i} - 2F_{2i}(s)\vartheta^{2i} + F_{2i}(s^-) (\vartheta^-)^{2i} \right) \in \vartheta \bQ_{\ge 2k+3}[\vartheta], 
\end{equation}
we try to solve 
\[
	 F_{2k+2}(s^+)(\vartheta^+)^{2k+2} - 2F_{2k+2}(s)\vartheta^{2k+2} + F_{2k+2}(s^-) (\vartheta^-)^{2k+2}  + R_{2k + 4} = O(\vartheta^{2k+6}). 
\]

We split
\begin{align}
	&F_{2k+2}(s^+)(\vartheta^+)^{2k +2} -2 F_{2k+2}(s)\vartheta^{2k+2} + F_{2k+2}(s^-)(\vartheta^-)^{2k+2} \nonumber\\
	& = \left( F_{2k+2}(s^+) -2 F_{2k+2}(s) + F_{2k+2}(s^-) \right) \vartheta^{2k+2} \label{eq:F-diff} \\
	&\quad + F_{2k+2}(s)\left( (\vartheta^+)^{2k+2} - 2\vartheta^{2k+2} + (\vartheta^-)^{2k+2}\right) \label{eq:z-diff}  \\
	&\quad + \left( F_{2k+2}(s^+) - F_{2k+2}(s) \right)\left( (\vartheta^+)^{2k+2} - \vartheta^{2k+2} \label{eq:z-diff2} \right)  \\
	& \quad + \left( F_{2k+2}(s^-) - F_{2k+2}(s) \right)\left( (\vartheta^-)^{2k+2} - \vartheta^{2k+2} \right). \nonumber 
\end{align}
The line \eqref{eq:F-diff} is equal to 
\[
\left( (2b_2)F_{2k+2}'(s)  + b_1^2 F_{2k+2}''(s) \right) \vartheta^{2k+2} + O(\vartheta^{2k + 6}),
\]
by the same computations as the $k = 0$ case. The line \eqref{eq:z-diff} is 
\[
	\begin{aligned}
		& (\vartheta^+)^{2k+2} - 2\vartheta^{2k+2} + (\vartheta^-)^{2k+2} \\
		& = \left( \vartheta + d_2 \vartheta^2 + d_3 \vartheta^3 + O(\vartheta^4) \right)^{2k+2} 
		+ \left( \vartheta - d_2 \vartheta^2 + d_3 \vartheta^3 + O(\vartheta^4) \right)^{2k+2} - 2 \vartheta^{2k+2} \\
		& = 2(2k+2)d_3 \vartheta^{2k+4} + (2k+2)(2k+1) d_2^2 \vartheta^{2k+4} + O(\vartheta^{2k + 6}). 
	\end{aligned}
\]
The two lines in \eqref{eq:z-diff2} is equal to 
\[
\begin{aligned}
& \left( F_{2k+2}'(s) (b_1 \vartheta + O(\vartheta^2)) \right)\left( (\vartheta + d_2 \vartheta^2 + O(\vartheta^3))^{2k+2} - \vartheta^{2k+2} \right) \\
&\quad+ \left( F_{2k+2}'(s) (- b_1 \vartheta + O(\vartheta^2)) \right)\left( (\vartheta - d_2 \vartheta^2+ O(\vartheta^3))^{2k+2} - \vartheta^{2k+2} \right) \\
& = 2(2k+2)b_1d_2 F_{2k+2}'(s) + O(\vartheta^{2k + 6}), 
\end{aligned}
\]
with $R^{(2)}_{2k + 6} = O(\vartheta^{2k + 6})$. 
Combining all of the above, we get 
\begin{equation}
  \label{eq:F2k2-remainder}
  \begin{aligned}
& F_{2k+2}(s^+)(\vartheta^+)^{2k +2} -2 F_{2k+2}(s)\vartheta^{2k+2} + F_{2k+2}(s^-)(\vartheta^-)^{2k+2} \\
& = \Bigl( b_1^2 F_{2k+2}'' + (2b_2 + 2(2k+2)b_1d_2)F_{2k+2}'  \\
& \quad \quad + (2(2k+2)d_3 + (2k+2)(2k+1)d_2^2) F_{2k+2} \Bigr) \vartheta^{2k+4} + O(\vartheta^{2k + 6}). 
\end{aligned}
\end{equation}
Therefore we can choose $F_{2k+2}$ to be the solution of 
\[
 	b_1^2 G'' + (2b_2 + 2(2k+2)b_1d_2)G'  + (2(2k+2)d_3 + (2k+2)(2k+1)d_2^2) G = - P_{2k+4},  
\]
where $P_{2k + 4} \in \bP_{2k + 3}(\rho)$ is the the coefficient to the $\vartheta^{2k + 4}$ term in $R_{2k + 4}$. Explicitly, the equation reads (denote $m = 2k+2$):
\[
4\rho^2 G'' - (m-1)\frac{8}{3} \dot{\rho}\rho G' + \left(m(m+1) \frac49 \dot{\rho}^2 - \frac43 m \ddot{\rho}\rho \right) G = - P_{m+2}. 
\]
The substitution $G = g \rho^{\frac{m}{3}}$ converts the equation to 
\[
3 \rho^2 g'' + 2 \dot{\rho} \rho g' = - \frac34 \rho^{-\frac{m-3}{3}}(P_{m+2} + 1) =: \hat{P}_{m+2}, 
\]
whose solution can be explicitly given by 
\[
g(s) = C_1\int_0^s \rho^{-\frac23}(\tau) d\tau +  \int_0^s \rho^{-\frac23}(\tau)\int_0^\tau \rho^{-\frac43}(\sigma) \hat{P}_{m+2}(\sigma) d\sigma d\tau, 
\]
and $G(s) = g \rho^{\frac{m}3}$. Here the constant $C_1$ is uniquely determined by periodicity, and note $\hat{P}_{m + 2} = \hat{P}_{2k + 4} \in \bP_{2k + 3}(\rho)$. 

We now note that if $P(s) \in \bP_{m}(\rho)$, then $P(s + h)- P(s) \in h^{-m}\bQ_{\ge m+1}$. Using this and the explicit formula defining $G(s) = F_{2k + 2}(s)$, (noting that we integrate twice) we get $G(s + h) - G(s) \in h^{-2k -1} \bQ_{\ge 2k + 2}$. Equipped with this fact, we revisit the remainders in the previous calculations. Since
\begin{multline*}
F_{2k+2}(s^+)(\vartheta^+)^{2k +2} -2 F_{2k+2}(s)\vartheta^{2k+2} + F_{2k+2}(s^-)(\vartheta^-)^{2k+2}  + P_{2k + 4} \\
  \in \vartheta^{2k + 2} \cdot \vartheta^{-2k-1} \bQ_{\ge 2k +2} = \vartheta \bQ_{\ge 2k + 2}, 
\end{multline*}
and also is of $O(\vartheta^{2k + 6})$, it must be contained in $\vartheta \bQ_{\ge 2k + 5}$, which is our inductive hypothesis \eqref{eq:inductive}. 
\end{proof}

\begin{proof}
[Proof of Proposition~\ref{prop:gen-laz}]
We truncate the coordinate change \eqref{eq:higher-lazutkin} up to order $2k$, and denote the coordinate change by $\Phi$. For  $A(x, y) = (x + y, y)$, then 
\[
\begin{aligned}
  & 	(\Phi \circ T - A \circ \Phi) (s, \vartheta)  = (Y(s^+, \vartheta^+) - Y(s, \vartheta), Y(s^+, \vartheta^+) - Y(s, \vartheta))  \\
  & \quad = (R(s, \vartheta) \vartheta^{2k+4}, R(s, \vartheta) \vartheta^{2k+4})
\end{aligned}
\]
where $R$ is an anlytic functions in $(s, \vartheta)$. The norm of $R$ can be estimated by the $(2k +4)$th derivatives of $\Phi \circ T - A \circ \Phi$. 

By Proposition~\ref{prop:formal-conj}, the $2k+4$th derivative of $R = X(s^+, \vartheta^+) - 2 X(s, \vartheta) + X(s^-, \vartheta^-)$ depends on up to $2k+4$ derivatives on $F_{2i}\vartheta^{2i}$. Since $F_{2i}''$ depends on $2i+1$ derivatives of $\rho$,  the norm of $R$ is estimated by up to $2k +3$ derivatives of $\rho$. We conclude using the Cauchy estimate that 
\[
\|R\|_{\sigma} \le C(k) (\sigma_0/2 - \sigma)^{-2k-3}\|\rho\|_{\sigma_0/2}. 
\]
Here the constant $C(k)$ comes from the actual differential polynomial in the coordinate change, and can be made explicit for every finite $k$. 
 By \eqref{eq:Phi-y}, the same estimate, with possibly a different constant, holds for $\Phi \circ T \circ \Phi^{-1} - A$. 
\end{proof}

\bibliographystyle{plain}
\bibliography{billiard}

\end{document}